\documentclass{amsart}

\usepackage{amssymb}
\usepackage{graphicx}

\usepackage{url}
\urldef{\mailsa}\path|{alfred.hofmann, ursula.barth, ingrid.haas, frank.holzwarth,|
\urldef{\mailsb}\path|anna.kramer, leonie.kunz, christine.reiss, nicole.sator,|
\urldef{\mailsc}\path|erika.siebert-cole, peter.strasser, lncs}@springer.com|    

\newtheorem{theorem}{Theorem}[section]
\newtheorem{lemma}[theorem]{Lemma}
\newtheorem{proposition}[theorem]{Proposition}

\theoremstyle{definition}
\newtheorem{definition}[theorem]{Definition}

\theoremstyle{remark}

\numberwithin{equation}{section}

\title{EL-Shelling on Comodernistic Lattices}

\author{Tiansi Li}
\address{}
\curraddr{}
\email{}
\thanks{Department of Mathematics and Statistics, Washington University in St. Louis}


\date{}

\dedicatory{}

\begin{document}

\begin{abstract}
We prove the equivalence of EL-shellability and the existence of recursive atom ordering independent of roots. We show that a comodernistic lattice, as defined by Schweig and Woodroofe, admits a recursive atom ordering independent of roots, therefore is EL-shellable. We also present and discuss a simpler EL-shelling on one of the most important classes of comodernistic lattice, the order congruence lattices.
\end{abstract}

\maketitle

\section{Introduction}

Modernistic and Comodernistic lattices are two large classes of finite lattices with shellable order complexes. Schweig and Woodroofe defined and studied these lattices in \cite{schweig2017broad} and showed that a wide range of lattices are either modernistic or comodernistic, including subgroup lattices of finite solvable groups, supersolvable and left-modular lattices, semi-modular lattices, $k$-equal partition lattices, order congruence lattices, and others. They proved in \cite{schweig2017broad} that comodernistic lattices are CL-shellable, which implies that the order complexes of modernistic and comodernistic lattices are shellable.

In this paper, we show that comodernistic lattices are EL-shellable, as defined in \cite{bjorner1983lexicographically}. This can be viewed as a generalization of Woodroofe's result on subgroup lattices. After Shareshian showed in \cite{shareshian2001shellability} that subgroup lattices of solvable groups are CL-shellable, Woodroofe proved a stronger result in \cite{woodroofe2008}  that these lattices are in fact EL-shellable.

It is shown in \cite{bjorner1983lexicographically} that a poset is CL-shellable if and only if it admits a recursive atom ordering. We will prove a similar result for EL-shellable posets in order to prove our main result.

Recall that a recursive atom ordering of $P$ assigns to each pair $(x,r)$ with $x$$\in$$P$$\backslash$$\{\hat{1}\}$ and a maximal chain $r$ in $[\hat{0}, x]$ an ordering of the atoms in $[x,\hat{1}]$

\begin{lemma}
If $P$ admits a recursive atom ordering such that, for each pair $(x,r)$, the atom ordering of $[x,\hat{1}]$ does not depend on $r$, then $P$ is EL-shellable.
\end{lemma}

The following is our main result.

\begin{theorem}
Comodernistic lattices are EL-shellable.
\end{theorem}

We will further study the class of order congruence lattices after proving Theorem 1.2. We will show that these lattices are EL-shellable using integer labels without invoking Theorem 1.2.

In the next section, we will include necessary definitions concerning modernistic and comodernistic lattices. For readers unfamiliar with EL-shellability and CL-shellability, we recommend \cite{bjorner1983lexicographically} for more information.

\section{Preliminaries}

We call a simplicial complex $\Delta$ $shellable$ if there is an ordering of its facets $F_1, F_2, \dots, F_t$ such that $(\bigcup_{i=1}^{k} F_i)\cap F_{k+1}$ is pure and $(dim(F_k)-1)$-dimensional for $k=1, 2, \dots, t-1$. A poset $P$ is $shellable$ if its order complex $\Delta P$ is.

For a poset $P$, an edge-labeling is a map from the edge set of the Hasse diagram of $P$ to some label poset $\Lambda$. An edge-labeling is an $EL-shelling$ if for every interval $[x, y]$ of $P$, there exists a unique weakly increasing maximal chain lexicographically preceding all other maximal chains in $[x, y]$ in a fixed linear extension of the poset $\Lambda$. A chain-edge labeling is a map from the set of all pairs $(c, e)$ to the label set $\Lambda$, where $c$ is a maximal chain of $P$ and $e$ is an edge in $c$, such that $(c, e)$ and $(c', e)$ get the same label if $c$ and $c'$ coincide from $\hat{0}$ to $e$. A poset $P$ is $CL-shellable$ if it admits a chain-edge labeling such that for every interval $[x, y]$ and maximal chain $r$ in $[\hat{0}, x]$, in the rooted interval $[x, y]_r$, there exists a unique weakly increasing maximal chain lexicographically preceding all other maximal chains in a fixed linear extension of the label poset $\Lambda$.

Clearly, EL-shellability implies CL-shellability. More importantly, both EL-shellability and CL-shellability imply shellabiliy.

\begin{theorem} \cite[Proposition 2.3]{bjorner1983lexicographically}
EL-shellability $\Rightarrow$ CL-shellability $\Rightarrow$ Shellability.
\end{theorem}

Another notion commonly used in the context of lexicographic shellability is recursive atom ordering.

\begin{definition} \cite[Definition 4.2.1]{wachs2006poset}
A poset $P$ is said to admit a recursive atom ordering if the length of $P$ is $1$ or if the length of $P$ is greater than $1$ and there is an ordering $a_1, a_2,\dots ,a_t$ of the atoms of P which satisfies:
\begin{enumerate} 
    \item For all $j = 1,2,\dots,t$, $[a_j, \hat{1}]$ admits a recursive atom ordering in which the atoms of $[a_j, \hat{1}]$ that come first in the ordering are those that cover some $a_i$, where $i<j$.
    \item For all $i<j$, if $a_i, a_j < y$ then there is a $k<j$ and an atom $z$ of $a_j$ such that $z<y$ and $a_k<z$.
\end{enumerate}
\end{definition}

Bj\"oner and Wachs have proved that CL-shellability is equivalent to the existence of a recursive atom ordering.

\begin{theorem} \cite[Theorem 4.2.2]{wachs2006poset}
A bounded poset $P$ admits a recursive atom ordering if and only if $P$ is CL-shellable.
\end{theorem}

Next we introduce modernistic and comodernistic lattices. Let $L$ denote a lattice. Recall that an element $m$ in $L$ is $left-modular$ if for any $x<y$ in $L$, we have $(x\vee m)\wedge y = x\vee (m\wedge y)$. A lattice $L$ is $modernistic$ if for every interval of $L$, there exists a left-modular atom in that interval. A lattice is $comodernistic$ if it is the dual of a modernistic lattice. That is, there exists a left-modular coatom in every interval. Schweig and Woodroofe proved that comodernistic lattices are CL-shellable \cite{schweig2017broad}.

\begin{theorem} \cite[Theorem 1.2]{schweig2017broad}
If $L$ is a comodernistic lattice, then $L$ has a CL-labeling
\end{theorem}

As stated above in Theorem 1.2, we will show that comodernistic lattices are EL-shellable. To show that comodernistic lattices are EL-shellable, we need the notion of a sub-M-chain, which can be viewed as an analogy of an M-chain in a left-modular lattice.

\begin{definition} \cite{schweig2017broad}
A maximal chain $\hat{0} = m_0\lessdot m_1\lessdot \dots \lessdot m_n =\hat{1}$ in $L$ is a sub-M-chain if for every $i$, the element $m_i$ is left-modular in the interval $[\hat{0}, m_{i+1}]$. 
\end{definition}

We also list here two lemmas from \cite{schweig2017broad} that will help prove our main theorem. We refer readers to \cite{schweig2017broad} for the proofs of the lemmas.

\begin{lemma} \cite[Lemma 3.1]{schweig2017broad}
Let $L$ be a lattice with a sub-M-chain $\textbf{m}$ of length $n$. Then no chain of $L$ has length greater than $n$.
\end{lemma}

\begin{lemma} \cite[Lemma 2.12]{schweig2017broad}
Let $m$ be a coatom of the lattice $L$. Then $m$ is left-modular in $L$ if and only if for every $y$ such that $y\nleq m$ we have $m\wedge y\lessdot y$.
\end{lemma}

\section{Proof of Lemma 1.1}

Let $P$ be a CL-shellable poset with an induced recursive atom ordering such that for each $x$$\in$$P$$\backslash$$\{\hat{1}\}$, all orderings of the atoms of $[x,\hat{1}]$ are the same, as in Lemma 1.1. We will construct an EL-labeling for $P$. For every edge $e=[x, y]$ in the Hasse diagram of $P$, we define a $(\kappa_e+1)$-tuple, where $\kappa_e$ is the number of maximal chains in $[\hat{0}, x]$. The last coordinate records the edge itself. That is, for $e=[x, y]$, let the last coordinate of the label of $e$ simply be the 2-tuple (x, y). Now we consider the first $\kappa_e$ coordinates. These are indexed by the roots in $[x,y]$. Call these the $\kappa_e$ large coordinates of the label. In each large coordinate, we place a $2$-tuple, whose first entry is the union of the corresponding root and the edge itself, and the second entry is the label induced by the root in the CL-shelling of $P$. Let us call these $2$-tuples the small coordinates. We order the $\kappa_e$ large coordinates according to the original CL-labeling of $P$. That is, if $C$ is lexicographically the $k^{th}$ maximal chain in $[\hat{0}, x]$ according to the CL-labeling, the $k^{th}$ large coordinate of $e$ consists of $C\cup e$ and the label induced by $C$.

Now we define a partial order on the labeling set. Suppose $e=[x, y]$ and $e'=[x', y']$ are two edges in the Hasse diagram labeled as above. Then we say $e\leq e'$ if $y<x'$, or if $y=x'$ and for some large coordinate $e_m$ of $e$, there exists a large coordinate $e'_n$ of $e'$ such that $e_m \leq e'_n$, by which we mean that the root (first small coordinate) of $e_m$ is contained in the root of $e'_n$, and the label (second small coordinate) in $e_m$ is less than or equal to the label in $e'_n$.

Let us first check that this is a well-defined partial order. Obviously we have reflexivity. Antisymmetry is satisfied because if $e < e'$, $x$ cannot be above or equal $y'$. Transitivity holds because if $e < e'$ and $e' < e''$, $y<x''$. So this is indeed a partial order.

We now check that this edge-labeling gives an EL-shelling of $P$. For any interval $[x, y]$, let $C$ be the weakly-increasing chain in the original CL-labeling (with respect to any roots). We claim that the new edge-labeling on $P$ makes $C$ the unique weakly increasing and lexicographic first maximal chain in the interval.

The fact that $C$ is weakly-increasing follows from the consistency of the CL-labeling. Let $C$ be weakly increasing in $[x,y]_r$ for all roots $r$. Suppose there exists some $C'=\{c_0 \lessdot c_1 \lessdot \dots \lessdot c_k\}$ in $[x,y]$ that is also weakly increasing in the new edge-labeling. Then for each $0<i<k$, there exists a root $r_i$ such that $c_i$ is the first atom in $[c_{i-1}, c_{i+1}]_{r_i}$. Notice that the assumption on the atom orderings implies that whether labels of two consecutive edges in $[c_{i-1}, c_{i+1}]$ are weakly increasing is independent of roots. Hence $C'$ must be weakly increasing in some rooted interval, which would further imply that there are two weakly increasing maximal chains in one rooted interval. This contradicts to $P$ being CL-shellable. A similar argument shows that $C$ must be the lexicographic first maximal chain of the interval. \qed

\bigskip

\section{Proof of Theorem 1.2}

Let $L$ be a comodernistic lattice. We fix a sub-M-chain $\textbf{m} = \{\hat{0} = m_0 \lessdot{m_1} \lessdot{\dots} \lessdot{m_n} = \hat{1}\}$ of $L$ and prove Theorem 1.2 by induction on height. For each atom $a$ of $L$, label $[\hat{0}, a]$ by the index of the minimal element of $\textbf{m}$ that is larger than or equal to $a$ in the partial order. Suppose we have labeled all elements of height less than $k$. For $y$ of height $k$, we pick an element $z$ covered by $y$ and define a sub-M-chain $\textbf{m}_y$ in $[y, \hat{1}]$ as follows. Let $w$ be the minimal element above $y$ in $\textbf{m}_z$, where $\textbf{m}_z$ has been chosen by induction, and pick some sub-M-chain of $[y, w]$. Then $\textbf{m}_y$ is the sub-M-chain of $[y, w]$ followed by the rest of the sub-M-chain of $\textbf{m}_z$. We label the atom edges of $[y, \hat{1}]$ using $\textbf{m}_y$ the same way we label atom edges of $L$ with $\textbf{m}$.

Now we show that the edge labeling defined above induces a recursive atom ordering independent of roots as in Lemma 1.1. We claim that any atom ordering that is consistent with the edge labeling is a recursive atom ordering. That is, if for a fixed linear extension, the label of $[x, y]$ precedes the label of $[x, y']$, where $y$ and $y'$ both cover $x$, then $y$ precedes $y'$ in the atom ordering of $x$. And if the label of $[x, y]$ equals the label of $[x, y']$, we can either have $y$ precedes $y'$ or $y'$ precedes $y$. We prove this claim by induction on coheight. It is obvious for all intervals $[x, \hat{1}]$ where $x$ is a coatom of $L$. Suppose the claim stands for all $[x, \hat{1}]$ of length less than $k$. For any $[x, \hat{1}]$ of length $k$, consider any atom $a_i$ of the interval. 

\begin{figure}
    \centering
    \includegraphics[width=1\columnwidth]{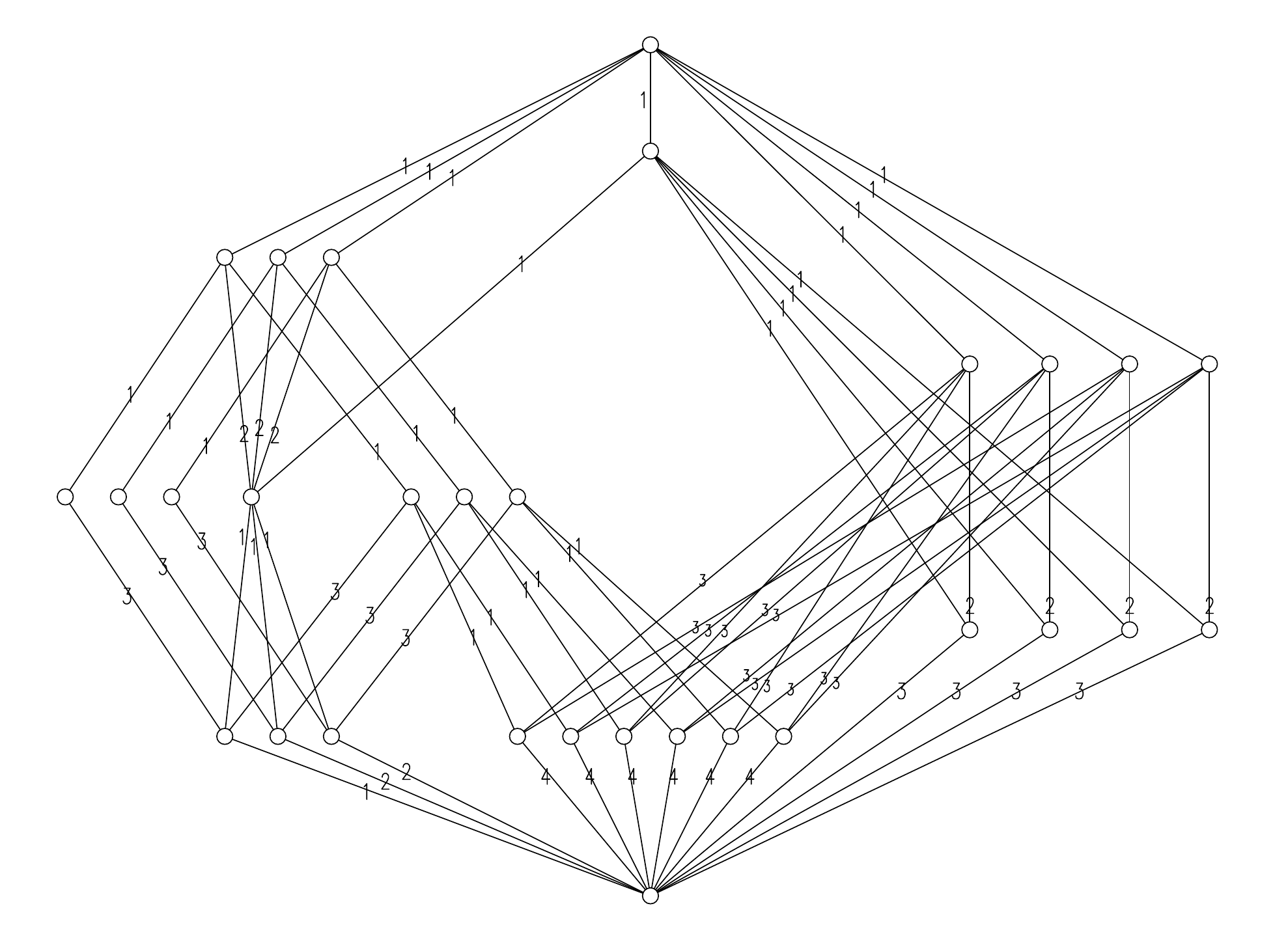}
    \caption{Initial labeling on $S_4$ given by sub-M-chain}
    \label{fig:my_label}
\end{figure}

\begin{figure}
    \centering
    \includegraphics[width=1\columnwidth]{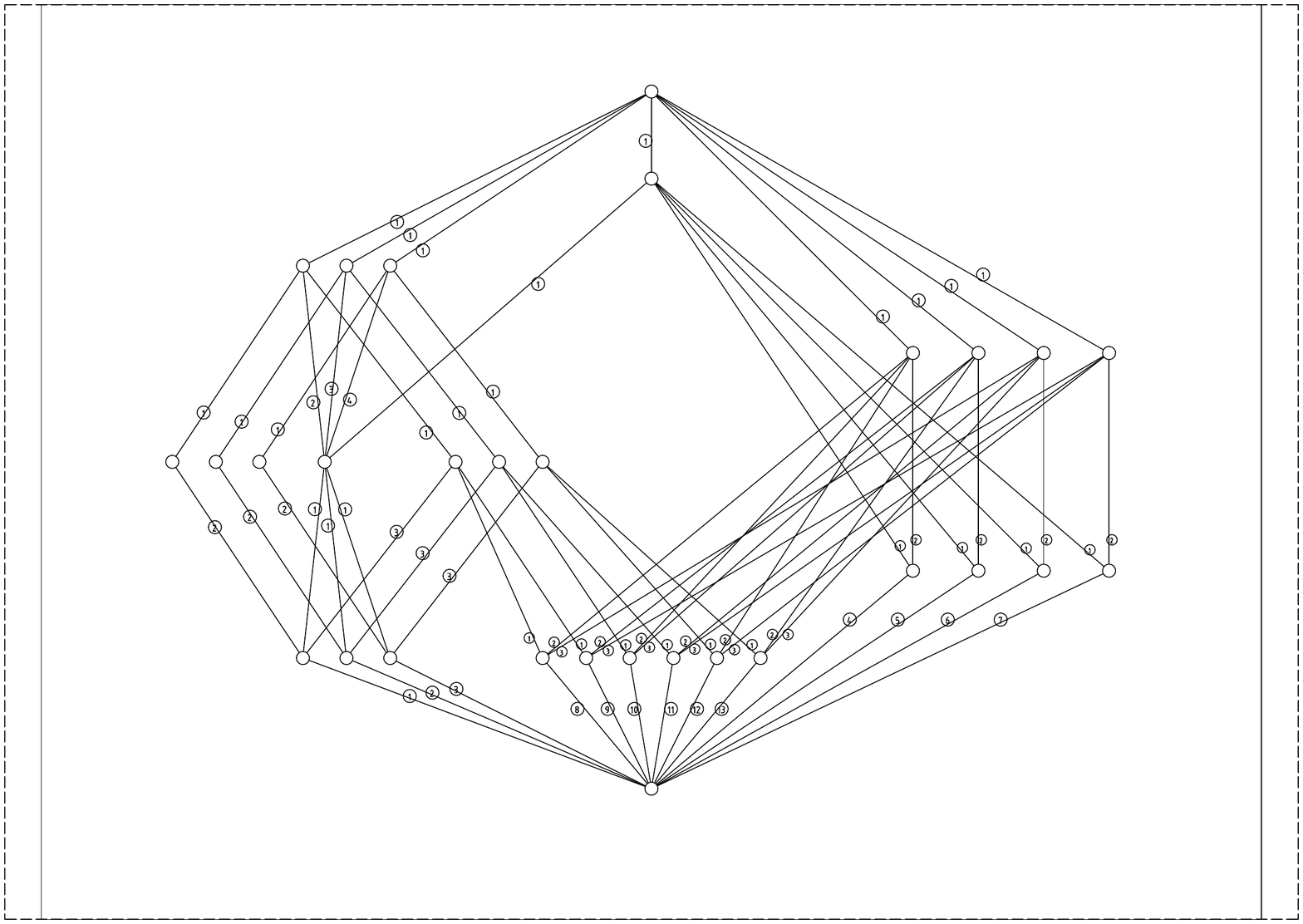}
    \caption{Root-independent recursive atom ordering from the initial labeling}
    \label{fig:my_label}
\end{figure}

$[a_i, \hat{1}]$ admits a recursive atom ordering by induction. For any two atoms $b_j$ and $b_k$ of $[a_i, \hat{1}]$, suppose $j<k$ and $b_k$ covers some $a_l$ with $l<i$. We need to show that $b_j$ covers some $a_m$ for $m<i$. Since $b_k$ covers $a_i$ and $a_l$ with $l<i$, $b_k \in [x, w]$ where $w$ is the minimal element above $a_i$ in the sub-M-chain $\textbf{m}_x$ of $[x, \hat{1}]$. Therefore $b_j < w$ since $j<k$. Let $w'$ be the maximal element in $\textbf{m}_x$ below $w$ and not above $b_j$. Then $b_j \wedge w'$ is an atom in $[x, \hat{1}]$ prior to $a_i$ and covered by $b_j$.

Now for any two atoms $a_i$ and $a_j$ with $i<j$ of $[x, \hat{1}]$, we need to find an atom of $[a_j, \hat{1}]$ that covers some $a_k < a_i \vee a_j$ with $k<j$. Consider the first atom $b$ in $[a_j, \hat{1}]$ that sits below $a_i \vee a_j$. Let $w$ be the minimal element of $\textbf{m}_x$ that sits above $a_i \vee a_j$. Then $w$ is the minimal element of $\textbf{m}_x$ that sits above $a_j$ since $i<j$. Let $w'$ be the element in $\textbf{m}_x$ covered by $w$. Then $b \wedge w'$ is an atom in $[x, \hat{1}]$ that is prior to $a_j$ \qed

\section{Order Congruence Lattices}

The order congruence lattice $\mathcal{O}(P)$ of a poset $P$ is the set of all equivalence classes of level set partitions from $P$ to $\mathbb{Z}$. That is, the set of all weakly order preserving maps from $P$ to $\mathbb{Z}$, where
two such maps are considered equivalent if they induce the same partition on $P$.

For example, the order congruence lattice on a totally ordered set is a boolean lattice. The order congruence lattice on a set of pairwise incomparable elements is isomorphic to a partition lattice. In general, order congruence lattice of any poset can be considered as in between the boolean lattice and the partition lattice.

Schweig and Woodroofe proved in \cite{schweig2017broad} that order congruence lattices are comodernistic, therefore CL-shellable. We here present a different proof where any linear extension of $P$ gives a sub-M-chain and an EL-shelling on $\mathcal{O}(P)$.

\begin{figure}
    \centering
    \includegraphics[width=1\columnwidth]{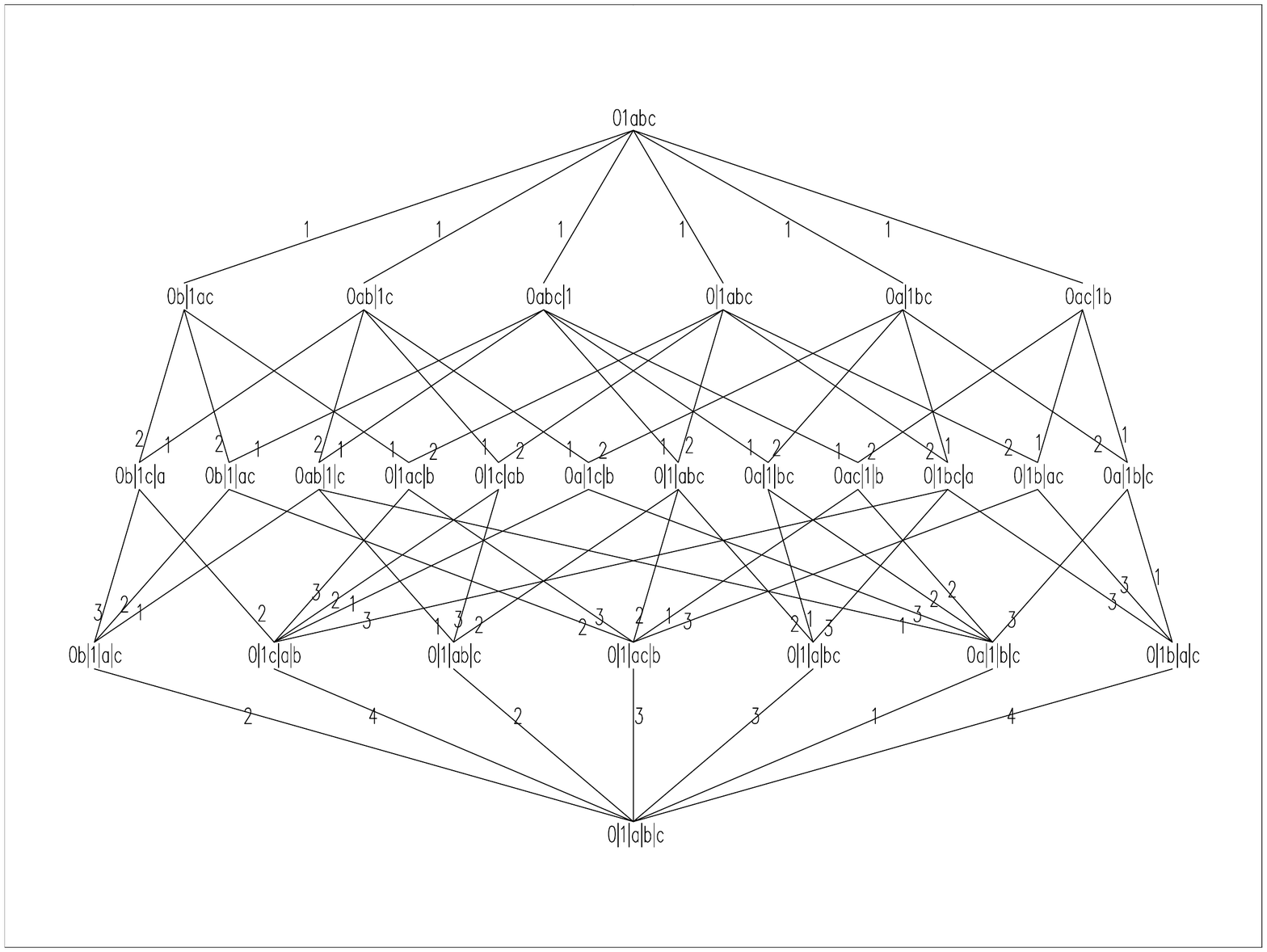}
    \caption{EL-shelling on $N_5$ with linear extension $\hat{0}\rightarrow a\rightarrow b\rightarrow c\rightarrow \hat{1}$}
    \label{fig:my_label}
\end{figure}

Fix a linear extension of $P=\{z_1, z_2, \dots, z_n\}$. For an element in $\mathcal{O}(P)$ with $k$ blocks, we can index each block as follows. First we take the smallest element in the linear extension of each block as representative. We assign indices $1$ through $k$ to the blocks preserving the order of their representatives. And we use these indices to construct a sub-M-chain and EL-shelling as follows.

Let $C$ be the maximal chain $\hat{0} = c_1\lessdot c_1\lessdot \dots \lessdot c_n =\hat{1}$ where $c_k$ is obtained by having the first $k$ elements ($z_1$ through $z_k$) in one block, and then one block each for the remaining elements.

\begin{proposition}
$C$ is a sub-M-chain which induces an EL-shelling on $\mathcal{O}(P)$.
\end{proposition}

\begin{proof}
For any $k$ and $x\in [\hat{0}, c_k]$, we need to show either $x<c_{k-1}$ or $x\wedge c_{k-1}\lessdot x$. If $z_k$ in $x$ is a block by itself, $x<c_{k-1}$. Otherwise, $x\wedge c_{k-1}$ is obtained by isolating $z_k$ to a single block from $x$. Hence $x\wedge c_{k-1}$ is covered by $x$ and $C$ is a sub-M-chain.

Now we show that it induces an EL-shelling on $\mathcal{O}(P)$. Consider the following edge-labeling. For any edge $[x, y]$, if $y$ is obtained from $x$ by merging the $i^{th}$ block and the $j^{th}$ block where $i<j$, we assign $j$ to $[x, y]$. This is a well-defined edge-labeling. Now we show that it is an EL-shelling.

Consider any interval $[x, y]$ in $\mathcal{O}(P)$. Notice that if $y$ consists of $k$ blocks, then any edge in the interval can be viewed as a merge within one of the $k$ blocks of $y$. The lexicographically first maximal chain is obtained by consecutively merging the smallest indexed two sub-blocks in the same block of $y$. This is a weakly increasing chain. Next we check that this is the unique weakly increasing chain of the interval. We show that any other chain must either violate the merging order within a block of $y$ or among blocks of $y$. Suppose we merge two sub-blocks within a block of $y$ that are not the two smallest possible blocks to merge. Merging with a smaller sub-block later will result in a smaller edge label, in which case the maximal chain cannot be weakly increasing. Suppose we missed a merge within a smaller block of $y$. The edge obtained by completing that missed merge will again create a smaller label hence the maximal chain cannot be weakly increasing. 
\end{proof}

\bibliographystyle{plain}
\bibliography{bibliography.bib}

\end{document}